\documentclass[preprint,12pt]{elsarticle}

\usepackage{amsfonts,amsthm,amsmath}
\usepackage{mathrsfs}

\usepackage{times}
\usepackage{amsfonts}
\usepackage{amsthm}
\usepackage{amsmath}
\usepackage{epic}
\usepackage{times}

\usepackage[margin=3cm]{geometry}
\usepackage{amsthm,amsmath}
\usepackage{fullpage}

\usepackage{tikz}
\usetikzlibrary{snakes}
\usepackage{soul}

\newcommand{\Nn}{{\mathbb N}}

\newcommand{\orb}{\rm orb}

\newcommand{\Ff}{{\rm Fact}}
\newcommand{\card}{{\rm card}}
\def\A{\mathbb{A}}
\def\P{\mathscr{P}}

\newtheorem{thm}{Theorem}

\newtheorem{rema}{Remark}

\newtheorem{lemma}[thm]{Lemma}
\newtheorem{corollary}[thm]{Corollary}
\newtheorem{proposition}[thm]{Proposition}
\newtheorem{question}{Question}
\newtheorem{conjecture}{Conjecture}

\newtheorem{example}{Example}
\newtheorem{definition}[thm]{Definition}
\theoremstyle{definition}

\begin{document}

\begin{frontmatter}

\title{On prefixal factorizations of words}

\author[label1]{Aldo de Luca}
  \ead{aldo.deluca@unina.it}

   \author[label2,label3]{Luca Q. Zamboni}
  \ead{lupastis@gmail.com}

\address[label1]{Dipartimento di Matematica e Applicazioni,
Universit\`a di Napoli Federico II, Italy}
\address[label3]{FUNDIM, University of Turku, Finland}
\address[label2]{Universit\'e de Lyon,
Universit\'e Lyon 1, CNRS UMR 5208,
Institut Camille Jordan,
43 boulevard du 11 novembre 1918,
F69622 Villeurbanne Cedex, France}

\begin{abstract}

We consider the class ${\cal P}_1$ of all infinite words $x\in \A^\omega$ over a finite alphabet $\A$ admitting a prefixal factorization, i.e., a  factorization $x= U_0 U_1U_2 \cdots $ where each  $U_i$ is a non-empty prefix of  $x.$ With each $x\in {\cal P}_1$ one naturally associates a ``derived" infinite word $\delta(x)$ which may or may not admit a prefixal factorization. 
We are interested in the class  ${\cal P}_{\infty}$ of all words $x$ of  ${\cal P}_1$ such that $\delta^n(x) \in {\cal P}_1$ for all $n\geq 1$. Our primary motivation for studying the class  ${\cal P}_{\infty}$ stems from its connection to a coloring problem on infinite words independently posed by T. Brown in \cite{BTC} and by the second author in \cite{LQZ}.  More precisely, let  ${\bf P}$ be the class of all words $x\in \A^\omega$ such that for every finite coloring  $\varphi : \A^+ \rightarrow C$ there exist $c\in C$ and a factorization $x= V_0V_1V_2\cdots  $ with $\varphi(V_i)=c$ for each $i\geq 0.$ In \cite{DPZ} we conjectured that a word $x\in {\bf P}$ if and only if $x$ is purely periodic. In this paper we show that 
  ${\bf P}\subseteq {\cal P}_{\infty},$ so in other words, potential candidates to a counter-example to our conjecture are amongst the non-periodic elements of  ${\cal P}_{\infty}.$ 
 We establish several results on the class  ${\cal P}_{\infty}$. In particular,  we show that a Sturmian word $x$ belongs to ${\cal P}_{\infty}$ if and only if  $x$ is nonsingular, i.e.,   no proper suffix of $x$  is a standard Sturmian word. 

\end{abstract}
\begin{keyword} Combinatorics on words, Prefixal factorization, Sturmian word, Coloring problems
\MSC[2010] 68R15
\end{keyword}

\journal{}

\end{frontmatter}
\section{Introduction}

 Let ${\bold P}$ denote the class of all infinite words $x\in \A^\omega$ over a finite alphabet  $\A$ such that for every finite coloring  $\varphi : \A^+\rightarrow C$ there exist $c\in C$ and a factorization  $x= V_0V_1V_2\cdots $ with $\varphi(V_i)=c$ for all $i\geq 0.$ Such a factorization is called $\varphi$-monochromatic.   In  \cite{DPZ} we conjectured:
 
 \vspace{2 mm}
 
\begin{conjecture}\label{concol}
 Let $x$ be an infinite word. Then  $x\in {\bold P}$ if and only if $x$ is (purely) periodic. 
 \end{conjecture}
 
 \vspace{2 mm}
 
\noindent 
 Various partial results in support of Conjecture~\ref{concol} were obtained in \cite{DPZ, DZ, ST}.    
Given $x\in \A^\omega,$ it is natural to consider the binary coloring $\varphi:\A^+\rightarrow \{0,1\}$ defined by $\varphi(u)=0$ if $u$ is a prefix of $x$ and $\varphi(u)=1$ otherwise. Then any $\varphi$-monochromatic factorization is nothing more than a prefixal factorization of $x,$ i.e., a factorization of the form $x=U_0U_1U_2\cdots $ where each $U_i$ is a non-empty prefix of $x.$ 
Thus a first necessary condition for a word $x$ to belong to  ${\bf P}$ is that $x$ admit a prefixal factorization.  
Not all infinite words admit such a factorization including for instance the class of  square-free words  and the class of Lyndon words \cite{DPZ}.

Thus in the study of the Conjecture~\ref{concol}, one can restrict to the class of words ${\cal P}_1$  admitting a prefixal factorization. But in fact more is true. 
It is shown that if $x\in  {\cal P}_1,$ then $x$ has only finitely many distinct unbordered prefixes and admits a unique factorization in terms of its unbordered prefixes. This allows us to associate with each $x\in  {\cal P}_1$ a new infinite word $\delta(x)$ on an alphabet corresponding to the finite set of   unbordered prefixes of $x.$ In turn, the word $\delta(x)$ may or may not  admit a prefixal factorization. In case $\delta(x)\notin {\cal P}_1,$
then $\delta(x) \notin  {\bold P}$ and from this one may deduce that $x$ itself does not belong to $ {\bold P}.$
This is for instance the case of the famous Thue-Morse infinite word 
$t =t_0t_1t_2\cdots \in \{0,1\}^\omega$  where $t_n$ is defined as the sum modulo $2$ of the digits  in the binary expansion of
$n,$ \[t = 011010011001011010010\cdots \]The origins of $t$ go back to the beginning of the last century with the works of  A. Thue \cite{Th1, Th2} in which he proves amongst other things that $t$ is {\it overlap-free}, i.e., contains no word of the form $uuu'$ where $u'$ is a non-empty prefix of $u.$ 
It is readily checked that $t$ admits a prefixal factorization, in particular  $t$  may be factored uniquely as $ t =V_0V_1V_2\cdots$ where each $V_i\in\{0,01,011\}.$ On the other hand as is shown later (see Example \ref{ex:tm}), the {\it derived word} $\delta(t)$ is the square-free ternary Thue-Morse word fixed by the morphism $1\mapsto 123,$ $2\mapsto 13,$ $3\mapsto 1.$ Hence  $\delta(t)\notin  {\cal P}_1.$ This in turn implies
that $t\notin {\bold P}.$ Concretely, 
consider the coloring $\varphi ': \{0,1\}^+\rightarrow \{0,1,2\}$ defined by $\varphi '(u)=0$ if $u$ is a prefix of $t$ ending with $0,$ $\varphi '(u)=1$ if $u$ is a prefix of $t$ ending with $1,$ and $\varphi '(u)=2$ otherwise. We claim that $t$ does not admit a $\varphi '$-monochromatic factorization. In fact, suppose to the contrary that $t =V_0V_1V_2\cdots$ is a  $\varphi '$-monochromatic factorization. Since $V_0$ is a prefix of $t$, it follows that there exists $a\in \{0,1\}$ such that each $V_i$ is a
  prefix of $t$ terminating with $a.$ Pick $i\geq 1$ such that $|V_i|\leq |V_{i+1}|.$ Then  $aV_iV_i \in \makebox{Fact}(t).$ Writing
$V_i=ua,$ (with $u$ empty or in $\{0,1\}^+),$ we see $aV_iV_i=auaua $ is an overlap, contradicting that $t$ is overlap-free.  

Thus, in the study of Conjecture~\ref{concol}, one can further restrict to the subset ${\cal P}_2$ of ${\cal P}_1$  consisting of all $x\in {\cal P}_1$ for which $\delta(x) \in {\cal P}_1.$ In this case, one can define a second derived word $\delta^2(x)=\delta (\delta(x))$ which again may or may not belong to ${\cal P}_1.$ In case $\delta^2(x) \notin {\cal P}_1,$ then not only is $\delta^2(x) \notin {\bf P},$ but as we shall see neither are $\delta(x)$ and $x.$
Continuing in this way, we are led to consider  the class  ${\cal P}_{\infty}$ of all words $x$ in ${\cal P}_1$ such that $\delta^n(x) \in {\cal P}_1$ for all $n\geq 1.$ We show that 
  ${\bf P}\subset {\cal P}_{\infty},$ so in other words any potential counter-example to our conjecture is amongst the  non-periodic words belonging to  ${\cal P}_{\infty}.$  However, $ {\bf P}\neq {\cal P}_{\infty}.$ In fact, we prove in Sect.~\ref{sec:sei}  that a large class of Sturmian words (nonsingular Sturmian words) belong to $ {\cal P}_{\infty},$ while as shown  in \cite{DPZ}, no Sturmian word belongs to ${\bf P}$.

The paper is organized as follows: In Sect.~\ref{sec:2} we give a brief overview of some basic definitions and notions in combinatorics on words which are relevant to the subsequent material.  In Sect.~\ref{sec:3} we study the basic properties of words admitting a prefixal factorization and in particular show each admits a unique factorization in terms of its finite set of unbordered prefixes.  From this we define the derived word $\delta(x).$   We prove amongst other things that if $x$ is a fixed point of a morphism, then the same is true of  $\delta(x)$. 

In Sect.~\ref{sec:quattro}
we recursively define a nested sequence $\cdots \subset {\cal P}_{n+1}\subset {\cal P}_{n} \subset \cdots \subset {\cal P}_{1} $ where  ${\cal P}_{n+1}=\{x\in {\cal P}_{n}\,|\, \delta(x) \in {\cal P}_n\},$ and study some basic properties of the set ${\cal P}_\infty =\bigcap_{n\geq 1}{\cal P}_n.$

In Sect.~\ref{sec:5} we study the connection between the class ${\bf P}$ and the class ${\cal P}_{\infty}$ and show that   ${\bf P}\subset {\cal P}_{\infty}.$  We also show  that if $x\in {\cal P}_{\infty}$, then $x$ is uniformly recurrent, from which we recover a result previously proved in \cite{DPZ} via different techniques.

Sect.~\ref{sec:sei} is devoted to prefixal factorizations of Sturmian words. Any Sturmian word $x\neq aS$, where $a\in \{0,1\}$ and $S$ a standard Sturmian word, admits a prefixal factorization. 
The main result of the section is  that  a Sturmian word $x$ belongs to ${\cal P}_{\infty}$ if and only if $x$ is nonsingular, i.e.,  no proper suffix of $x$ is a standard Sturmian word.

\section{Notation and Preliminaries}\label{Prel}\label{sec:2}

Given a  non-empty set $\A,$ or {\em alphabet},  we let  $\A^*$ denote the set of
all finite words $u=u_1u_2\cdots u_n$ with $u_i\in \A.$ The
quantity $n$ is called the {\em length} of $u$ and is denoted $|u|.$
The {\em empty word}, denoted $\varepsilon,$ is the
unique element in $\A^*$ with $|\varepsilon|=0.$
We set
$\A^+=\A^*\setminus \{\varepsilon\}.$ For each word $v\in \A^+$, let $|u|_v$  denote the number of occurrences
of $v$ in $u$. In the following we suppose that the alphabet $\A$ is finite even though several results hold true
for any alphabet.

Let $u\in \A^*$. A word $v$ is a \emph{factor} of $u$ if there exist words $r$ and $s$ such that $u=rvs$; $v$ is a \emph{proper} factor if $v\neq u$. 
If $r= \varepsilon$ (resp., $s=\varepsilon$), then $v$ is called a \emph{prefix} (resp.,~a \emph{suffix}) of $u$.

Given words $u,v\in \A^+$ we say $v$ is a {\it border} of $u$ if $v$ is both a proper prefix and a proper suffix of $u.$ In case $u$ admits a border, we say $u$ is {\it bordered}. Otherwise $u$ is called {\it unbordered}. 

  Let  $\A^\omega$  denote the set of all
one-sided infinite words $x=x_0x_1\cdots $ with $x_i\in \A$, $i\geq 0$. 

Given $x\in \A^\omega,$ let
$\Ff^+ (x) =\{x_ix_{i+1}\cdots x_{i+j}\,|\, i,j\geq 0\}$ denote the set of all non-empty {\it factors} of $x$. Moreover, we set
$\Ff (x) =  \{\varepsilon\} \cup \Ff^+ (x)$.
The {\em factor complexity} of $x$ is the map $\lambda_x: \Nn\rightarrow \Nn$ defined as follows: for any $n\geq 0$
$$ \lambda_x(n) = \card( \A^n \cap \Ff (x)),$$
i.e., $\lambda_x(n)$ counts the number of distinct factors of $x$ of length $n$.
A  factor $u$ of a finite or infinite word  $x$  is called {\em right special} (resp., {\em left special}) if there exist two different letters $a$ and $b$  such that $ua$ and $ub$ (resp., $au$ and $bu$) are factors of $x$. A factor $u$ of $x$ which is right and left special is called {\em bispecial}.\

Given $x=x_0x_1x_2\cdots \in \A^\omega.$ 
A factor $u$ of  $x \in \A^\omega$ is called {\em recurrent} if $u$ occurs in $x$ an infinite number of times, and is called  {\em uniformly recurrent} if 
there exists an integer $k$ such that every  factor of $x$ of length $k$ contains an occurrence of $u$. An infinite word $x$ is called {\it recurrent} (resp., {\it uniformly recurrent}) if each of its factors is recurrent (resp., uniformly recurrent).

Let $x \in \A^{\omega}$ and ${\cal S}$ denote the {\em shift operator}. The {\em shift orbit} of $x$ is the set $\orb(x)=\{{\cal S}^k(x) \mid k\geq 0\}$,  i.e., the set of all suffixes of $x$. The {\em shift orbit closure} of $x$ is the set $\Omega(x)= \{y \in \A^{\omega} \mid \Ff  (y) \subseteq \Ff (x )\}$.

 An infinite word $x$ is called (purely) {\it periodic} if $x=u^\omega$ for some $u\in A^+,$ and is called {\it ultimately periodic} if $x=vu^\omega$ for some $v\in \A^*,$ and $u\in \A^+.$ As is well known,  an ultimately periodic word which is non-periodic is not recurrent. The word $x$ is called {\it aperiodic} if  $x$ is not ultimately periodic.
 
 We say that two finite or infinite words $x=x_0x_1\ldots $ and $y=y_0y_1\ldots $ on the alphabets $\A$ and $\A'$ respectively  are {\em word isomorphic}, or simply {\em isomorphic},  and write $x\simeq y,$
if there exists a bijection $\phi: \A\rightarrow \A'$ such that $y=\phi(x_0)\phi(x_1)\ldots.$ 

  For all definitions and notation not explicitly given in the paper, the reader is referred to the books  \cite{AS,LO,LO2}.

\section{Prefixal factorizations}\label{sec:3}

\begin{definition} We say that an infinite word $x\in \A^{\omega}$ admits  a {\em prefixal factorization}
if $x$ has a factorization
$$ x= U_0U_1U_2\cdots $$
where each $U_i$, $i\geq 0$, is a non-empty prefix of $x$.
\end{definition}

Some  properties of words having a prefixal factorization have been proved in \cite{DPZ}. We mention in particular the
following:

  \begin{lemma}\label{lemma:B} Let $x\in A^{\omega}$ be an infinite word having a prefixal factorization. Then the first letter of  $x$ is uniformly  recurrent.
\end{lemma}

Given $x=x_0x_1x_2\cdots \in \A^\omega,$ we let $UP(x)$ denote the set of all (non-empty) unbordered prefixes of $x.$

\begin{proposition}\label{PF} Let $x=x_0x_1x_2\cdots \in \A^\omega.$ The following conditions are equivalent:

\begin{enumerate}
\item $x$ admits a prefixal factorization.

\item $x$ admits a unique factorization of the form $x=U_0U_1U_2\cdots $ with $U_i\in UP(x)$ for each $i\geq 0.$ 

\item $\card(UP(x))<+\infty.$ 

\end{enumerate}
\end{proposition}

\begin{proof}
Let us first prove that if $x$ admits a factorization $x=U_0U_1U_2\cdots $ with $U_i\in UP(x),$ then such a factorization is necessarily unique. 
Indeed, suppose that there
 exists a different  factorization $ x = U'_0U'_1U'_2\cdots $ with $U'_i\in UP(x).$ Let  $n\geq 0$ be the first integer such that     $U_n \neq U'_n$.  Without loss of generality we suppose that $|U_n|>|U'_n|$. We can write $U_n= U'_nU'_{n+1}\cdots U'_{n+p}\xi$, with $p\geq 0$ and $\xi$ prefix of $U'_{n+p+1}.$  Hence, $U_n$ is  bordered,  a contradiction.
 
  We will  now show that $3.\Rightarrow 2.\Rightarrow 1.\Rightarrow 3.$ 
  
$3.\Rightarrow 2$. We begin by assuming $\card(UP(x))<+\infty$ and show how to construct a factorization of $x$ in terms of unbordered prefixes of $x.$   We define recursively an infinite sequence
$U_0,U_1,U_2,\ldots \in UP(x)$ such that  $U_0U_1\cdots U_n$ is a prefix of $x$ for each $n\geq 0,$ $U_0$ is the longest unbordered prefix of $x,$ and for $n\geq 1,$  $U_n$ is the longest unbordered prefix  of $x$ which is a prefix of  $(U_0\cdots U_{n-1})^{-1}x$.  For $n=0$ we simply set $U_0$ to be the longest unbordered prefix of $x.$ Note $U_0$ is well defined since $\card(UP(x))<+\infty.$ For the inductive step, let $n\geq 0$ and suppose we have defined $U_0,\ldots ,U_n$ with the required properties. We show how to construct $U_{n+1}.$ Let $V$ be the prefix of $x$ of length $|U_0\cdots U_n|+1.$ Then since $|V|>|U_0|$ it follows that $V$ is bordered. Let $v$ denote the shortest border of $V.$  Then $v\in UP(x)$ and by induction hypothesis that $U_n$ is unbordered it follows that $|v|=1.$ In other words, $(U_0 \cdots U_n)^{-1}x$ begins with  an unbordered prefix of $x.$ Thus we define $U_{n+1}$ to be the longest unbordered prefix of $x$ which is a prefix of $(U_0 \cdots U_{n})^{-1}x.$
It follows immediately that $U_0 \cdots U_{n}U_{n+1}$ is a prefix of $x.$ Thus we have shown that  $3.\Rightarrow 2.$. 

 $2.\Rightarrow 1.$ This implication is trivially true. 
 
 $1.\Rightarrow 3.$  If $x=V_0V_1V_2\cdots$ is a prefixal factorization of $x,$ then each prefix of $x$ longer than $|V_0|$ is necessarily bordered. Hence, $\card(UP(x))\leq |V_0|.$ 
\end{proof} 

A direct proof of the equivalence of conditions 1. and 3. in the preceding proposition is in \cite[Lemma 3.7]{DPZ}. We also observe that an infinite word having a finite number of unbordered factors is purely periodic \cite{HHZ}.

Let $\P_1$ denote the set of all infinite words $x=x_0x_1x_2\cdots $ over any finite alphabet satisfying 
any one of the three equivalent conditions given in Proposition~\ref{PF}. 
For $x\in \P_1$ let
\begin{equation}\label{*} x=U_0U_1U_2\cdots \end{equation}
be the unique factorization of $x$ with  $U_i\in UP(x)$ for $i\geq 0.$ Let $UP'(x)=\{U_i\,|\,i\geq 0\}\subseteq UP(x),$ and set $n_x=\card(UP'(x)).$

Given distinct elements $U,V\in UP'(x),$ we write $U \prec V$ if  
\[\min\{i\,|\, U_i=U\}<\min\{i\,|\, U_i=V\},\]
in other words if the first occurrence of $U$ in (1) is before the first occurrence of $V$ in (1).
Let 
\[\phi: \{1,2,\ldots ,n_x\}\rightarrow UP'(x)\] denote the unique order preserving bijection. We define $\delta (x) \in \{1,2,\ldots ,n_x\}^\omega$ by
\[\delta(x)=\phi^{-1}(U_0)\phi^{-1}(U_1)\phi^{-1}(U_2)\cdots .\]
Clearly $\phi (\delta(x))=x.$ We call  $\delta(x)$ the {\em derived word} of $x$ with respect to the morphism induced by the bijection $\phi: \{1,2, \ldots, n_x\} \rightarrow UP'(x)$.

\begin{example}\label{fibe} \rm{ Let $\A =\{0, 1\}$ and let $f$ be the Fibonacci word over $\A$, 
$$ f= 010010100100101001010010010100\cdots, $$
which is fixed by the morphism (Fibonacci morphism) defined by $0\mapsto 01$, $1\mapsto 0$. It is readily verified that $UP(f)= UP'(f)= \{0, 01\}$ and that
$01\prec 0$. One has $n_f= 2$ and $\phi(1)= 01, \phi(2) =0$. The unique factorization of $f$ in terms of $UP(f)$ is
$$f = (01)(0)(01)(01)(0)(01)(0)(01)(01)(0)(01)(01)(0)(01)(01)(0) \cdots$$
Hence,
$$\delta(f)= 1211212112112121121 \cdots,$$
and $\delta(f) \simeq f$.
}
\end{example}

\begin{example}\label{tribe}\rm{Let $x=121312112131212131211213121312112131212131211213\cdots$
denote the Tribonacci word fixed by the morphism defined by $1\mapsto 12, 2\mapsto 13, 3\mapsto 1.$
It is readily verified that $UP'(x)=UP(x)=\{1,12,1213\},$ and  $1213 \prec 12 \prec 1.$ It follows that
$n_x=3$ and $\phi(1)=1213,$ $\phi(2)=12,$ and $\phi(3)=1.$ The unique factorization of $x$ in terms of $UP(x)$ begins with
\[x=(1213)(12)(1)(1213)(12)(1213)(12)(1)(1213)(1213)(12)(1)(1213)(12)(1213)(12)(1)(1213)\cdots\]
and hence 
\[\delta(x)=123121231123121231\cdots.\]
}
\end{example}
\begin{rema}\label{01fib}{\rm In general, if  $x\in {\cal P}_1$, the set  $UP'(x)$ may be a proper subset of $UP(x).$ For instance, consider $x=10f$ where $f$ is the Fibonacci word. Then it is readily verified that $UP(x)=\{1,10,100\}$ while $UP'(x)=\{10,100\}.$ }
\end{rema}

We extend $\phi$ to a morphism $\phi:\{1,2,\ldots ,n_x\}^+\rightarrow UP'(x)^+.$

\begin{lemma}\label{injective} The morphism  $\phi:\{1,2,\ldots ,n_x\}^+\rightarrow UP'(x)^+$ is injective.
\end{lemma}

\begin{proof}Suppose $w=\phi(v)=\phi (v')$ with $v,v' \in \{1,2,\ldots ,n_x\}^+.$ Then $w$ factors as a product of elements in $UP'(x).$ Since any such factorization is necessarily unique, it follows that $v=v'.$
\end{proof}

While, as is readily verified,  every prefix $w$ of $x$ may be written uniquely as a product of unbordered prefixes of $x,$ in general, as we saw in the example of $10f$ (see Remark \ref{01fib}), it may not be possible to factor $w$ over $UP'(x).$ However, the following lemma shows that if $w$ occurs in a prefixal factorization of $x,$ then $w=\phi(v)$ for some factor $v$ of $\delta(x).$

\begin{lemma}\label{fact} Let $x=V_0V_1V_2\cdots$ be a prefixal factorization of $x.$ Then there exists a (unique) factorization $\delta(x)=v_0v_1v_2\cdots$ such that $\phi(v_i)=V_i$ for each $i\geq 0.$
\end{lemma}

\begin{proof} Let $x= U_0U_1U_2 \cdots $ be the factorization of $x$ in unbordered prefixes. Define $r_0=0$ and $r_n=\sum_{i=0}^{n-1}|U_i|.$ In other words, $r_n$ corresponds to the position of $U_n$ in the preceding factorization. Similarly we define $s_0=0$ and $s_n=\sum_{i=0}^{n-1}|V_i|.$
Then we claim that $\{s_n\,|\,n\geq 0\}\subseteq \{r_n\,|\,n\geq 0\}.$  In fact, suppose to the contrary
that there exist indices $i,j$ such that $r_j<s_i<r_{j+1}.$ This implies that there exists $k\geq i,$ such that a prefix of $V_k$ (possibly all of $V_k)$ is a proper suffix of $U_j.$ This is a contradiction since $U_j$ is unbordered.
Thus we have shown that any prefixal factorization of $x$ is also a factorization of $x$ viewed as an infinite word over the alphabet $UP'(x),$ in other words. The result now follows. 
\end{proof}

\noindent Combining the two previous lemmas we obtain:

\begin{corollary}Let $x=V_0V_1V_2\cdots$ be a prefixal factorization of $x.$ Then for each $i\geq 0$ there exists a unique factor $v_i$ of $\delta(x)$ such that $\phi(v_i)=V_i.$ 
\end{corollary}

\noindent As another consequence:

\begin{corollary} Suppose $x\in \P_1$ is a fixed point of a morphism $\tau:\A^+\rightarrow \A^+.$ Then there exists a morphism $\tau': \{1,2,\ldots ,n_x\}^+\rightarrow \{1,2,\ldots ,n_x\}^+$ fixing $\delta(x)$  such that
$\phi\circ \tau'=\tau\circ \phi.$ 
\end{corollary}

\begin{proof}  Applying $\tau$ to the unique factorization  $x=U_0U_1U_2\cdots$ with $U_i\in UP'(x),$ we obtain a prefixal factorization $x=\tau(U_0)\tau(U_1)\tau(U_2)\cdots.$  
Writing $\delta(x)=a_0a_1a_2\cdots$ with $a_i\in \{1,2,\ldots ,n_x\}$ and $\phi (a_i)=U_i,$ by Lemma \ref{fact} there exists a unique factorization $\delta(x)=v_0v_1v_2\cdots$ such that $\phi(v_i)= \tau(U_i)= \tau(\phi(a_i))$ for each $i \geq 0$.
The result now follows by defining $\tau'(a_i)=v_i.$ 
\end{proof}

\begin{example}\label{ex:try}\rm{As we saw in Example \ref{tribe}, the Tribonacci word $x$ is in $\P_1.$ It follows from the previous corollary that $\delta(x)$ is also a fixed point of a morphism $\tau':\{1,2,3\}^+\rightarrow \{1,2,3\}^+$ which we can compute using the relation $\tau'=\phi^{-1}\circ \tau \circ \phi$ where $\tau$ denotes the Tribonacci morphism defined by  $1\mapsto 12, 2\mapsto 13, 3\mapsto 1.$
So 
\[1\overset{\phi}\mapsto 1213\overset{\tau}\mapsto 1213121\overset{\phi^{-1}}\mapsto 123\]
\[2\overset{\phi}\mapsto 12\overset{\tau}\mapsto 1213\overset{\phi^{-1}}\mapsto 1\]
\[3\overset{\phi}\mapsto 1\overset{\tau}\mapsto 12\overset{\phi^{-1}}\mapsto 2.\]
Thus $\delta(x)$ is fixed by the morphism defined by $1\mapsto 123, 2\mapsto 1, 3\mapsto 2.$ 
It may be verified that in this example, $\delta(x)$ has the same factor complexity as $x,$ namely it contains $2n+1$ distinct factors of each length $n\geq 1.$ However, unlike $x$ which has a unique right and left special factor of each length, $\delta(x)$ has a unique left special factor of each length, and two right special factors of each length.  It is also readily verified that $\delta(x)\in \P_1.$ In fact $UP'(\delta(x))=UP(\delta(x))=\{1,12,123\},$ and $123\prec 12\prec 1.$ Thus we obtain the infinite word $\delta^2(x) =\delta(\delta(x))\in \{1,2,3\}^\omega$ which is a fixed point of a morphism $\tau''$ verifying $\tau''=\phi'^{-1}\circ \tau' \circ \phi'.$ We compute $\tau''$ as before:
\[1\overset{\phi'}\mapsto 123\overset{\tau'}\mapsto 12312\overset{\phi'^{-1}}\mapsto 12\]
\[2\overset{\phi'}\mapsto 12\overset{\tau'}\mapsto 1231\overset{\phi'^{-1}}\mapsto 13\]
\[3\overset{\phi'}\mapsto 1\overset{\tau'}\mapsto 123\overset{\phi'^{-1}}\mapsto 1\]
and hence $\tau''=\tau,$ whence $\delta^2(x)=x.$ Thus for each $n\geq 1$ we have that $\delta^n(x)\in \P_1$ and $\delta^n(x)=x$ for $n$ even  and $\delta^n(x)=\delta(x)$ for $n$ odd.}
\end{example}

\begin{rema}{\em  We note  that by Lemma \ref{lemma:B}, if $x\in {\cal P}_1$,   the first letter  $x^F$ of $x$ is uniformly recurrent in $x$, so that  one can also define (see \cite{Du}) the bijection $\sigma : \{1, \ldots, \card({\cal R}_{x^F})\} \rightarrow {\cal R}_{x^F}$, where ${\cal R}_{x^F}$ is the finite set of the first returns of $x^F$ to $x^F$ in $x$, and define the derived word $D_{x^F}(x)$ with respect to the morphism induced by $\sigma$ \cite{DPZ}. The two derived words $\delta(x)$ and  $D_{x^F}(x)$ can be equal, as in the case of  $x$ equal to Fibonacci word; they can be different as
in the case of Tribonacci word.  In the case of  a word $aS$ where  $a\in \{0, 1\}$ and $S$ is a standard Sturmian word, one has that $\delta(aS)$ is not defined (cf. Lemma \ref{rem:sturm}), whereas $D_a(aS)$ is defined.}
\end{rema}

\section{A hierarchy of words admitting a prefixal factorization}\label{sec:quattro}

We may recursively define a nested collection of words $\cdots \subseteq\P_{n+1}\subseteq \P_n\subseteq \cdots \subseteq \P_1$ by 
\[\P_{n+1}=\{x\in \P_n\,|\, \delta(x)\in \P_n\}\] and
\[\P_\infty =\bigcap_{n=1}^\infty\P_n.\]
Hence, a word $x\in {\cal P}_{\infty}$ if and only if  $x\in {\cal P}_{1}$ and $\delta^n(x) \in {\cal P}_{1}$ for all $n \geq 1$.
The previous example showed that the Tribonacci word belongs to $\P_\infty.$ Similarly, following Example \ref{fibe}, the Fibonacci word also belongs to $\P_\infty.$ In contrast, the following example shows that the Thue-Morse word does not belong to $\P_\infty.$

\begin{example}\label{ex:tm}\rm{The Thue-Morse word $t=0110100110010110\cdots $ is  fixed by the morphism
$\tau$ defined by $0\mapsto 01, 1\mapsto 10. $ It is readily verified that $t\in \P_1.$ In fact,
$UP'(t)=UP(t)=\{0, 01, 011\}$ and $011\prec 01\prec 0.$ Let $\phi: \{1,2,3\}\rightarrow UP'(t)$ be given by $1\mapsto 011, 2\mapsto 01, 3\mapsto 0.$ Then $\delta(t)$ is the fixed point of the morphism $\tau'$ which we compute as in Example \ref{ex:try}:
\[1\overset{\phi}\mapsto 011\overset{\tau}\mapsto 011010\overset{\phi^{-1}}\mapsto 123\]
\[2\overset{\phi}\mapsto 01\overset{\tau}\mapsto 0110\overset{\phi^{-1}}\mapsto 13\]
\[3\overset{\phi}\mapsto 0\overset{\tau}\mapsto 01\overset{\phi^{-1}}\mapsto 2.\]
We thus obtain that $\tau'$ is defined by $1\mapsto 123, 2\mapsto 13, 3\mapsto 2$ which is the well known Hall morphism \cite{H}. Thus $\delta(t)$ is the so-called {\it ternary Thue-Morse} word which is well known to be square-free (cf. \cite{AS}, \cite{LO}). It follows (cf. \cite{DPZ}) that  $\delta(t)$ does not admit a prefixal factorization, i.e., $\delta(t)\notin \P_1.$ Thus $t\notin \P_2$ and hence $t\notin \P_\infty.$ }
\end{example}

 Let $\leq_p$ denote the prefixal order in $\A^*$, i.e., for $u,v \in \A^*$, $u\leq_p v$ if $u$ is a prefix of $v$. We write
$u<_pv$ if $u$ is a proper prefix of $v$. For any word $u\in \A^+$ we let $u^F$ denote the first letter of $u$.

\begin{lemma} Let $\Gamma = \{u_1, u_2, \ldots, u_k\}$ be a finite set of unbordered words over the alphabet $\A =\{1, 2, \ldots, k\}$
such that  $u_k<_p u_{k-1} <_p \cdots <_p u_1$. Let $u_1^F= 1$ and $f$ be the morphism defined by  $i \mapsto u_i$, $i=1, \ldots, k$.
Then the word $x = f^{\omega}(1)$, fixed point of $f$, is such that $\delta(x)= x$.
\end{lemma}
\begin{proof} Let $x= f^{\omega}(1)= x_0x_1\cdots x_n \cdots$. Since $x$ is fixed by $f$, one has
$$ x= f(x)= f(x_0)f(x_1)\cdots  .$$
Hence, $x$ admits a unique factorization in  unbordered prefixes of the set $\{f(x_i) \mid i\geq 0\} = UP'(x)\subseteq \Gamma$. Let $\phi: \{1,2, \ldots, n_x\} \rightarrow UP'(x)$ be the unique order preserving bijection. One has:
$$ \delta(x) = \phi^{-1}(f(x_0))\phi^{-1}(f(x_1)) \cdots $$
Since $\phi^{-1}(f(x_n))=x_n$, $n\geq 0$, it follows $\delta(x)= x$.
\end{proof}

\begin{proposition}  For each $n\geq 1$ one has that ${\cal P}_{n+1}$ is properly included in ${\cal P}_n$. 
\end{proposition}
\begin{proof} Let  $t$  be the Thue-Morse word on two symbols  $t= 0110100110010110\cdots$. We have previously seen (see Example \ref{ex:tm}) that $t\in {\cal P}_1\setminus {\cal P}_2$. Let $F$ be the Fibonacci morphism and $F(t)= 01000100101000101001001\cdots$. The word $F(t)$ has a prefixal factorization and $\delta(F(t))\simeq t \in {\cal P}_1\setminus {\cal P}_2$.  It follows that $F(t) \in {\cal P}_2\setminus {\cal P}_3$.  It is easily verified that for any $n>1$ one has that $\delta(F^n(t)) \simeq F^{n-1}(t)$. From this it easily follows that $\delta^{n}(F^n(t))\simeq t$ and $\delta^{n+1}(F^n(t)) = \delta(t)\not\in {\cal P}_1$. Hence, $F^n(t)\in {\cal P}_{n+1}\setminus  {\cal P}_{n+2}$.
\end{proof}

Given an infinite word $x$ with $\card(UP(x))<\infty$, we let $N(x)$ denote the length of the longest unbordered prefix of $x$.
Now, for $x\in {\cal P}_\infty$, we define the map $\nu_x:\Nn \rightarrow \Nn$  by $$ \nu_x(n)=N(\delta^n(x)), \ n\geq 0,$$ where $\delta^0(x)=x$. This is well defined and the sequence $(\nu_x(n))_{n\geq 0}$ is a sequence of natural numbers ($\geq 2$)  which may be bounded or unbounded. If $x$ is the Tribonacci word we have $(\nu_x(n))_{n\geq 0} = 4, 3, 4, 3, 4, 3, \ldots$ 

 Let $a\in \A= \{0,1\}$ and put $b=1-a.$ In the following for  $a\in \{0, 1\}$, we let  $L_a$ be the injective endomorphism of $\{0, 1\}^*$ defined by
 \begin{equation}\label{eq:LA}
L_a: a\mapsto a,  b\mapsto ab.
\end{equation}

\begin{proposition} If  $x\in {\cal P}_\infty$ and $(\nu_x(n))_{n\geq 0}=2,2,2,2,....$,  then $x$ is isomorphic to the Fibonacci word. 
\end{proposition}

\begin{proof} Without loss of generality we can assume that $x$ begins with $0$. Since $N(x)=2$, it follows that $x$  begins with  $01$ and $01$  is the longest unbordered prefix of $x$. It follows that $x$ is a concatenation of $01$ and $0$ so we can write $ x=L_0(x')$ for some binary word $x' $ beginning with  $1$, where $x' $ is isomorphic to $\delta(x)$. Since $N(\delta(x))=N(x')=2$, it follows that $x' $ begins with $10$ and $10$ is the longest unbordered prefix of $x'$. Hence $x'$ is a concatenation of $10$ and $1$ so that  $x'=L_1(x'')$ where $x''$ begins with $ 0$ and $x''$ is  isomorphic to $\delta(x')=\delta^2(x)$.
Since $N(x'')=2$, it follows that $x''$  begins with $ 01$ and $01$ is the longest unbordered prefix of $x''$. Thus
$x''=L_0(x''')$ for some $x'''$ beginning with  $1$  and isomorphic to $\delta^3(x). $ Continuing in this way
 for each $k$ we have $(L_0L_1)^k(0)$ is a prefix of $x$. Hence,  $x$ is isomorphic to Fibonacci word.
\end{proof}

Let us observe that in general the sequence $(\nu(n))_{n\geq 0}$ does not determine $x$. For instance,
 let  $x$  be the  word fixed by the morphism $0\mapsto 00001$, $1\mapsto 0$
and $y$  be the  word fixed by the morphism $0\mapsto 00101$, $1\mapsto 001$. Then 
$(\nu_x(n))_{n\geq 0}= (\nu_y(n))_{n\geq 0}= 5,5,5,....$

\vspace{2 mm}

The following question naturally arises:

\begin {question} {\em What can be said about the nature of $ x$  if $(\nu_x(n))_{n\geq 0}$ is ultimately periodic?
Is $x$  necessarily a fixed point of a morphism? Conversely, if $x \in {\cal P}_\infty$ is a fixed point of a primitive morphism, is $(\nu_x(n))_{n\geq 0}$  ultimately periodic (in particular bounded) ?}
\end{question}

\section{A coloring problem}\label{sec:5}

  Let ${\bold P}$ be the class of all infinite words $x$ over any  finite alphabet $\A$  such that for every finite coloring $\varphi : \A^+ \rightarrow C$ there exists $c\in C$ and a factorization  $x= V_0V_1V_2\cdots  $ with $\varphi(V_i)=c$ for all $i\geq 0.$ Such a factorization is called $\varphi$-{\em monochromatic}.
  Thus if  $x\notin {\bold P}$, then there exists a finite coloring $\varphi: \A^+\rightarrow C$ such that for every factorization $x=V_0V_1V_2\cdots$ we have $\varphi(V_i)\neq \varphi(V_j)$ for some $i\neq j.$ 
Any such coloring  will be called a {\em separating coloring} for $x$.
  
  We conjectured \cite{DPZ}  that ${\bold P}$ coincides with the set of all periodic words. Partial results in this direction are given in \cite{DPZ} (see also \cite{DZ, ST}).
  
  \begin{lemma} Let $x\in {\cal P}_1$. The following holds:
  $$x \in {\bf P} \Longleftrightarrow \delta(x) \in {\bf P}. $$
  \end{lemma}
  \begin{proof} We begin by showing that if  $\delta(x) \in {\bf P}$ then  $x \in {\bf P}.$ Since   $x\in {\cal P}_1$,  we have $x= \phi(\delta(x))$, where $\phi$ is the morphism induced by the unique order bijection $\phi:\{1,2, \ldots, n_x\}
  \rightarrow UP'(x).$ Thus if $x\notin {\bf P},$ then by the morphic invariance property [3, Proposition 4.1], one obtains that $\delta(x) \notin {\bf P}.$ 
  
We next prove the converse. 
  Suppose $\delta(x) \notin {\bf P}.$ Then  there exists a separating coloring $\varphi: \A^+ \rightarrow C$ for $\delta(x)$.
Put $C'=C\cup\{+\infty, -\infty\}$ where we assume $+ \infty, -\infty \notin C.$ The coloring $\varphi$ induces a coloring
$\varphi': \A^+ \rightarrow C'$ defined as follows: 

\[ \varphi'(u) = \begin{cases}
\varphi (\phi^{-1}(u) )& \text{if $u$ is a prefix of $x$ and $u=\phi(v)$ for some $v\in \makebox{Fact}(\delta(x)) $}; \\ +\infty &  \text{if $u$ is a prefix of $x$ and $u\notin \phi(\makebox{Fact}(\delta(x)))$};
\\ -\infty &  \text{if $u$ is not a prefix of $x.$}
\end{cases} \]
We note that it follows from Lemma \ref{injective}  that $\varphi'$ is well defined. We now claim that  $\varphi'$ is a separating coloring for $x.$ Suppose to the contrary that $x$ admits a $\varphi'$-monochromatic factorization $x=V_0V_1V_2\cdots $ in non-empty factors.  Since $V_0$ is a prefix of $x$, it follows that $\varphi'(V_0)\neq -\infty,$ and hence $\varphi'(V_i)\neq -\infty$ for each $i\geq 0.$ Thus the factorization 
$x=V_0V_1V_2\cdots $ is a prefixal factorization of $x.$ It now follows from Lemma \ref{fact}  that there 
 exists a factorization $\delta(x)=v_0v_1v_2\cdots$ such that $\phi(v_i)=V_i$ for each $i\geq 0.$
 Since $\varphi'(V_i)=\varphi(v_i)$, it follows that $\delta(x)$ admits a $\varphi$-monochromatic factorization, a contradiction. 
\end{proof}

\begin{thm}\label{cor:piinfty} The following holds:  $ {\bf P}\subset{\cal P}_{\infty}.$
\end{thm}
\begin{proof} Suppose that  $x\not\in {\cal P}_{\infty}$. Thus there exists some $n\geq 1$ such that $x\not\in {\cal P}_n$. First suppose $x\not\in{\cal P}_1$. Since $x$ does not admit a prefixal factorization,  one has  $x\notin {\bf  P}$ (see \cite[Proposition~ 3.3]{DPZ}). Next suppose $x\in {\cal P}_1$ but $x\not\in {\cal P}_n$ for some
$n\geq 2$. Then $\delta^n(x)\not\in {\cal P}_1$ and hence as above $\delta^n(x) \notin {\bf P}$. By an iterated application of the preceding lemma it follows that $x\notin {\bf P}$.
\end{proof}

Let us observe that  if $x$ is a periodic word of $\A^{\omega}$, then  for every finite coloring  $x$ has a monochromatic factorization, so
that by  Theorem \ref{cor:piinfty}, or as one immediately verifies, $x\in {\cal P}_{\infty}$. From Theorem \ref{cor:piinfty} one has that  any counter-example to our Conjecture~\ref{concol} belongs to the set ${\cal P}_{\infty}$, which is our main motivation for studying this class of words. However, the converse of Theorem~\ref{cor:piinfty} is not true. For instance, as proved in \cite{DPZ}, Sturmian words do not belong to ${\bf P}$ whereas,
as we shall see in the next section (see Theorem \ref{Sturmsing}), a large class of Sturmian words (nonsingular Sturmian words) belong to $ {\cal P}_{\infty}$.

\begin{lemma}
 If $\varphi: \A^+ \rightarrow C$ is a separating coloring for $x$, then ${\hat \varphi} : \A^+\rightarrow C\cup \{\infty\}$, with $\infty \not\in C$, defined by 
$$ {\hat \varphi}(u)= \begin{cases}
\infty & \text{if $u$ is not a prefix of $x$}; \\
\varphi(u) & \text{if $u$ is a prefix of $x$,}
\end{cases}$$
is a separating coloring for $x$.
\end{lemma}
\begin{proof} Suppose that there exists a ${\hat \varphi}$-monochromatic factorization $$x = V_0V_1V_2\cdots $$ in non-empty factors $V_i$, $i\geq 0$. Since $V_0$ is a prefix of $x$, the preceding factorization has to be a  $\varphi$-monochromatic prefixal factorization, a contradiction as $\varphi$ is separating for $x$.
\end{proof}

\begin{proposition} Let $x\in \A^{\infty}$ and  $\Omega(x)$  the shift-orbit closure of $x$. If $x\not\in {\cal P}_{\infty}$, there exists a separating coloring $\varphi$ for $x$ such that if $y\in \Omega(x)$ and $y\neq x$,
then $y$ has a $\varphi$-monochromatic factorization.
\end{proposition}
\begin{proof}  By Theorem \ref{cor:piinfty} one has  $x\notin {\bf P}$. If $y\in \Omega(x)$ and $y\neq x$, then, as $x$ is not periodic, $y$ can  always  be factorized as
$y= V_0V_1V_2\cdots $ where each $V_i$, $i\geq 0$, is not a prefix of $x$. Let ${\hat \varphi}$ be the separating coloring for $x$
 defined in the  preceding lemma. Then one has that ${\hat \varphi}(V_j)= \infty$ for all $j\geq 0$.
\end{proof}

\begin{proposition}Let $x\in \A^\omega$ and let $\mathcal{A}$ be any  finite collection of words in ${\cal P}_\infty ^c\cap \Omega(x)$, where  ${\cal P}_\infty ^c$ denotes the complement of ${\cal P}_\infty.$ Then there exists a finite coloring $\varphi:\A^+\rightarrow C$ such that for each $y\in \Omega(x)$,   $\varphi$ is a separating coloring for $y$  if and only if $y \in \mathcal{A}.$
\end{proposition}
\begin{proof} The result is trivial if $\mathcal{A} =\emptyset$. Indeed in this case it is sufficient to consider the coloring $\varphi: \A^+\rightarrow C$ defined as follows: for any $u\in \A^+$, $\varphi(u)= c\in C$. In this way any $y\in \Omega(x)$ will have a $\varphi$-monochromatic factorization. Let us then suppose that $\mathcal{A}$ is not empty.

 Let $\mathcal{A} =\{y_1, \ldots, y_r \}$. Since for $1\leq i < j\leq r$, $y_i\neq y_j$, there exists a positive integer $k$ such
that any word of $\A^*$ of length $\geq k$ can be prefix of at most one of the words of  $\mathcal{A}$.  As  $\mathcal{A} \in {\cal P}_\infty ^c$ by Theorem \ref{cor:piinfty},   no word $y_i\in \mathcal{A}$ belongs to ${\bf P}$. Hence, for each  $1\leq i \leq r$ there exists a coloring $\varphi_i: \A^+ \rightarrow C_i$ which is separating for $y_i$. Let us observe that as $\mathcal{A}\subseteq \Omega(x)$ one has
$$\bigcup_{i=1}^r \Ff(y_i) \subseteq \Ff(x).$$
We can define a finite coloring $\varphi$ on $\A^+$ as follows. For $u\in \A^+$,

$$\varphi(u)= \begin{cases}
u & \text{ if $|u|<k$ and $u$ is a prefix of at least one word of $\mathcal{A}$; }\\
\varphi_1(u) & \text{ if  $u$ is a prefix of $y_1$ of length $\geq k$;}\\
\vdots             &    \vdots\\
\varphi_r(u) & \text{ if  $u$ is a prefix of $y_r$ of length $\geq k$;}\\
\infty & \text{if $u$ is not a prefix of any of the words of $\mathcal{A}$.}
\end{cases}$$

Let us first prove that for each $y\in \mathcal{A}$, the coloring $\varphi$ is separating. Let $y=y_i\in \mathcal{A}$ and suppose that
there exists a $\varphi$-monochromatic factorization  $y_i = V_0V_1V_2\cdots $ where each $V_i$ is non-empty.  Since $V_0$ is a prefix of $y_i$, the preceding factorization has to be a prefixal factorization. If $\varphi(V_i) = \varphi(V_0)$ and $|V_0|<k$ it would follow that $y_i= V_0^{\omega}$ a contradiction because $y_i\notin{\bf P}$. Thus as $V_0$ is a prefix of $y_i$ of length $\geq k$, it follows that $\varphi_i(V_j)= \varphi_i(V_0)$
for all $j\geq 0$ and this contradicts the fact that $\varphi_i$ is separating for $y_i$.

Let us now prove that if $y\not\in \mathcal{A}$, then $y$ admits a  $\varphi$-monochromatic factorization.  Since for each  $1\leq i\leq r$, $y\neq y_i$ and  $y_i$ is not periodic, one easily derives that $y$ can be factorized as $y = V_0V_1\cdots $, where each $V_j$,
$j\geq 0$, is not prefix of any of the words $y_i\in \mathcal{A}$, $i=1,\ldots,r$.  From this one has $\varphi(V_j)= \infty$ for all $j\geq 0$. 
\end{proof}

\begin{question} Let $x\in \A^\omega$ be uniformly recurrent.  Given a finite coloring $\varphi:\A^+\rightarrow C$ does there exist a finite (possibly empty) set $\mathcal{A}\subset \Omega(x)$ such that  for each $y\in \Omega(x)$ we have that $y$ admits a $\varphi$-monochromatic factorization if and only if $y\notin \mathcal{A}?$
\end{question}

Let us observe that in the previous question the hypothesis that $x$ is uniformly recurrent is necessary.  Indeed, let $x$ be word
$x= 010^21^20^31^3 \cdots$ and  $\varphi: \{0, 1\}^+\rightarrow \{0, 1, *\}$ the finite coloring defined for all  $u\in \Ff^+(x)$ by  $\varphi(u)= 0$, if $u$ begins with $0$,  $\varphi(u)=1$ if $u$ begins with $1$, and  if   $u\not\in \Ff^+(x)$, by $\varphi(u)= *$ .  In the shift  orbit closure $\Omega(x)$ of $x$ there are infinitely many words
$0^n1^{\omega}$, $n\geq 1$ which do not admit a $\varphi$-monochromatic factorization.

\begin{proposition}\label{prop:ur} Let $x\in \P_\infty.$ Then $x$ is uniformly recurrent. 

\end{proposition}

\begin{proof}We show by induction on $n\geq 1$ that for any infinite word $x$ over a finite alphabet if the prefix of $x$ of length $n$ is not uniformly recurrent, then $x\notin \P_\infty.$ If the first letter of $x$ is not uniformly recurrent in $x,$ then by Lemma \ref{lemma:B}  the word $x$ does not admit a prefixal factorization, hence $x\notin \P_1.$ 

Let $n\geq 1,$ and suppose the result holds up to $n.$ Let $x\in \A^\omega$ and suppose that the prefix $u$ of $x$ of length $n+1$ is not uniformly recurrent in $x. $ We will show that $x\notin \P_\infty.$  If $x\notin \P_1,$ we are done. So we may assume that $x\in \P_1.$ Let $\phi: \{1,2,\ldots ,n_x\}\rightarrow UP'(x)$ denote the unique order preserving bijection. Consider $\delta(x)\in  \{1,2,\ldots ,n_x\}^\omega.$ Since
$x=\phi(\delta(x))$ it follows that $\delta(x)$ is not  uniformly recurrent. If $v$ is a prefix of $\delta(x)$ which is uniformly recurrent in $\delta(x),$ then $\phi(v)$ is a uniformly recurrent prefix of $x.$ Moreover,  for every prefix $v$ of $\delta(x)$ we have $|v|<|\phi(v)|.$ Thus the shortest non-uniformly recurrent prefix of $\delta(x)$ is of length smaller than  or equal to $n.$ By induction hypothesis, $\delta(x)\notin \P_\infty,$ and hence $x\notin \P_\infty.$ 
\end{proof} 

\noindent As a consequence of  Proposition \ref{prop:ur} and Theorem \ref{cor:piinfty} we recover the following result first proved in \cite{DPZ}:

\begin{corollary} If  $x\in {\bf P}$ then $x$ is uniformly recurrent. 
\end{corollary}

\section{The case of Sturmian words}\label{sec:sei}

A word $x\in \{0,1\}^{\omega}$ is called {\it Sturmian} if  it is aperiodic and {\em balanced}, i.e., for all factors $u$ and $v$ of $x$ such that $|u|=|v|$ one has
$$ | |u|_a-|v|_a| \leq 1, \ a\in \{0,1\}.$$

\begin{definition} Let $a\in \{0,1\}$. We say that a Sturmian word is of type $a$  if it contains the factor $aa$.
\end{definition}
\noindent Clearly a Sturmian word is either of type $0$ or of type $1$, but not both.

 Alternatively, a binary infinite word $x$ is Sturmian if  $x$ has a unique left (or equivalently right) special factor  of length $n$  for each  integer $n\geq 0$. In terms of factor complexity, this is equivalent to saying  that  $\lambda_x(n)= n+1$ for $n\geq 0$. As a consequence one derives that  a Sturmian word $x$ is {\em closed under reversal}, i.e., if $u$ is a factor of $x$, then so is its reversal $u^{\sim}$ (see, for instance, \cite[Chap. 2]{LO2}).

 A Sturmian word $x$ is called {\em standard} (or {\em characteristic})  if  all its prefixes are left special factors of $x$. Since, as is well known,
 Sturmian words are uniformly recurrent, it follows that for any Sturmian word $x$ there exists a standard Sturmian word $S$ such that
 $\Ff (x) = \Ff (S)$.

 Following \cite{BPZ}  we say that a Sturmian word $x\in \{0,1\}^\omega$ is {\it singular} if it contains a standard Sturmian word as a proper suffix,
i.e., there exist  $u\in \{0, 1\}^+$ and a standard Sturmian word $S$ such that  $x=uS$. It is readily verified that  the previous factorization  is unique. A Sturmian word which is not singular  is said to be {\it nonsingular}. 
 
 Let $a\in \{0,1\}$ and  $b=1-a.$ In the following for  $a\in \{0, 1\}$, we consider the injective endomorphism  $L_a$ of $\{0, 1\}^*$  defined in (\ref{eq:LA}) and 
 the injective endomorphism $R_a$ of $\{0, 1\}^*$ defined by
$$  R_a : a\mapsto a,  b\mapsto ba .$$
We recall \cite[Chap. 2]{LO2} that the monoid generated by $L_a$ and $R_a$, with $a\in \{0,1\}$ contains all endomorphisms $f$ of $\{0, 1\}^*$ which preserve Sturmian words, i.e., the image $f(y)$ of any  Sturmian word $y$ is a Sturmian word.

Let us observe  that for any infinite   word $x$ over $\{0, 1\}$ we have $L_a(x)=aR_a(x).$ If  $x$ is a Sturmian word and $v$ is a left special factor of $x,$ then $L_a(v)$ is a left special factor of $L_a(x).$ Thus if $S$ is a standard Sturmian word, then so is $L_a(S)=aR_a(S).$  Conversely, if $S$ is a standard Sturmian word of type $a,$ then $S=L_a(S')$ for some standard Sturmian word $S'.$ 

\begin{lemma}\label{rem:sturm} A Sturmian word $x$  belongs to ${\cal P}_1$ if and only if $x\neq aS$ where $a\in \{0, 1\}$ and $S$ a standard  Sturmian word.
\end{lemma}
\begin{proof} Indeed, as is well known (see \cite{BDL, HN})  the unbordered factors of length greater than $1$ of a Sturmian word $x$ are of the form $bUc$ with $\{b, c\}= \{0, 1\}$  and $U$ a bispecial factor of $x$. Therefore, if $x = ax'$ has infinitely many unbordered prefixes, then $x'$ begins with infinitely many bispecial factors of $x,$ and hence $x'$ is a standard  Sturmian word. Conversely, let $x= aS$ where $S$ is a standard Sturmian word and $a\in \{0, 1\}$. The Sturmian word $aS$ does not admit a prefixal factorization. Indeed, $aS$ begins with  an infinite number of distinct prefixes of the form $aUb$ with $U$ bispecial and then a palindrome. It follows that $aS$ begins with arbitrarily long unbordered prefixes and hence, by Proposition \ref{PF} does not admit a prefixal factorization.
\end{proof}

We begin by reviewing some terminology which will be used in the proof of the following lemma. Let $x\in \A^\omega$ and $a\in \A.$ A word $u\in \A^+$ is called a {\it left first return} to $a$ in $x$ if $ua$ is a factor of $x$ which begins and ends with $a$ and $|ua|_a=2,$ i.e., the only occurrences of $a$ in $ua$ are as a prefix and as a suffix.  A word $u\in \A^+$ is called a {\it right first return} to $a$ in $x$ if $au$ is a factor of $x$ which begins and ends with $a$ and $|au|_a=2.$ A word $u\in \A^+$ is called a {\it complete return} to $a$ in $x$ if $u$ is a factor of $x$ which begins and ends with $a$ and $|u|_a\geq 2.$ It is a basic fact that if $u$ is a complete return to $a$ in $x,$ then $ua^{-1}$ factors uniquely as a product of left first returns to $a$ in $x$ and that $a^{-1}u$ factors uniquely as a product of right first returns to $a$ in $x.$

\begin{lemma}\label{ind}  Let $x\in \{0,1\}^\omega$ be a Sturmian word of type $a$ with $\card(UP(x))<+\infty$, and $N(x)$  the length of the longest unbordered prefix of $x.$   Then: 

\begin{enumerate}
\item[i)] If $N(x)=2,$ there exists a Sturmian word $y\in \{0,1\}^\omega$  isomorphic to $\delta(x)$ and  $x=L_a(y).$
\item[ii)] If $N(x)>2$ and $x$ begins with $a,$ there exists a Sturmian word $y\in \{0,1\}^\omega$ beginning with  $a$ such that  $x=L_a(y),$ $N(y)<N(x),$ and $\delta(x)=\delta(y).$   Moreover $L_a$ establishes a one-to-one  correspondence between $UP(y)$ and $UP(x),$ i.e., $L_a:UP(y)\rightarrow UP(x)$ is a bijection. 

\item[iii)] If $N(x)>2$ and $x$ begins with $b$ there exists a Sturmian word $y\in \{0,1\}^\omega$ beginning with $b$ such that  $x=R_a(y),$ $N(y)<N(x),$ and $\delta(x)=\delta(y).$ Moreover $R_a$ establishes a one-to-one correspondence between $UP(y)$ and $UP(x)\setminus \{b\},$ i.e., $R_a:UP(y)\rightarrow UP(x)\setminus \{b\}$ is a bijection.
\end{enumerate}
\end{lemma}

\begin{proof} To prove  i), suppose $N(x)=2.$ It follows that $UP(x)=\{c,cd\}$ where $\{c,d\}=\{0,1\}.$ Since $x$ is of type $a$ and admits a factorization over $UP(x)$ and hence contains the factor $cc$, it follows that $a=c$ and hence $UP(x)=\{a,ab\}.$   Thus the unique factorization of $x$ over $UP(x)$ is equal to the factorization of $x$ according to left first returns to $a$ which is well known to be Sturmian. Alternatively, there exists a unique Sturmian word $y$ such that $x=L_a(y).$ It follows that $\delta(x)$ is word isomorphic to $y.$

To prove   ii),  suppose $x$ begins with $a$ and its longest unbordered prefix is of length $N(x)>2.$  Then $x$ begins with $aa;$ in fact if $x$ begins with $ab,$ since $x$ is of type $a$ we would have that $UP(x)=\{a,ab\}$ contradicting our assumption that $N(x) >2.$ Having established that $x$ begins with $aa,$ it follows that there exists a unique Sturmian word $y$ such that $x=L_a(y)$ and moreover $y$ also begins with $a.$ In fact, $y$ is obtained from $x$ by factoring $x$ according to left first returns to $a$ where one codes the left first return $a$ by $a,$ and the left first return $ab$ by $b.$

Next we show that $L_a$ establishes a bijection between $UP(y)$ and $UP(x).$ We use the key fact that
if $u\in\{0,1\}^+$  and $ua$ is a factor of $x$ which begins and ends with $a,$ then there exists a unique factor 
$v$ of $y$ such that $u=L_a(v).$ In fact, $ua$ is a complete return to $a$ and hence $v$ is obtained from $u$ by factoring $u$ as a product of left first returns to $a.$ In particular, if  $u\in \{0,1\}^+$ is a factor of $x$ which begins with $a$ and ends with  $b,$ then $u=L_a(v)$ for some factor $v$ of $y.$

We begin by showing that $L_a(UP(y))\subseteq UP(x),$ i.e., that $L_a:UP(y)\rightarrow UP(x).$ 
So let $u$ be an unbordered prefix of $y.$ 
If $|u|=1,$ then $u=a$ and hence $L_a(u)=a$ which is an unbordered prefix of $x.$  If $|u|>1,$ then $u$ begins with $a$ and ends with $b,$ and hence $L_a(u)$ begins with $a$ and ends with  $b.$ If $L_a(u)$ were bordered, then any border $v$ of $L_a(u)$ would also begin with $a$ and end with $b,$ whence we can write $v=L_a(v')$ for some border $v'$ of $u,$ contradicting that $u$ is unbordered. 
Since the mapping $L_a:UP(y)\rightarrow UP(x)$ is clearly injective, to show that it is a bijection it remains to show that the mapping is surjective. 
So assume $u$ is a unbordered prefix of $x;$ we will show that 
$u=L_a(u')$ for some unbordered prefix of $y.$ This is clear if $u=a,$ in which case $u'=a.$ If $|u|>1,$ then $u$ begins with $a$ and ends with  $b$ and hence $u=L_a(u')$ for some prefix $u'$ of $y.$ Moreover, if $u'$ were bordered (say $v'$ is a border of $u')$ then $L_a(v')$ is a border of $u,$ a contradiction.

It follows that if $u$ is the longest unbordered prefix of $y,$ then $N(y)=|u|<|L_a(u)|\leq N(x)$ where the first inequality follows from the fact that $u$ must contain an occurrence of both $0$ and $1.$ Finally, since
$\card(UP(y))<+\infty,$ we have that $y$ admits a factorization $y=U_0U_1U_2\cdots$ with $U_i\in UP(y)$ which by definition is isomorphic to $\delta(y).$ 
Applying $L_a$ we obtain $$x=L_a(U_0)L_a(U_1)L_a(U_2)\cdots$$ with $L_a(U_i)\in UP(x)$ isomorphic to $\delta(x).$  Hence, $\delta(x)=\delta(y)$ as required. This completes the proof of ii).

Finally to prove  iii), suppose $x$ begins with $b.$  Then there exists a unique Sturmian word $y$ such that $x=R_a(y)$ and moreover $y$ begins with $b.$ As in the previous case, it is readily checked that $R_a:UP(y)\rightarrow UP(x)\setminus \{b\}$ is a bijection. The idea is that if $u\in\{0,1\}^+$  and $au$ is a factor of $x$ which begins and ends with $a$,  then there exists a unique factor 
$v$ of $y$ such that $u=R_a(v).$ In fact, $au$ is a complete return to $a$ and hence $u$ factors uniquely as a product of right first returns to $a.$ Thus in particular,  if  $u\in \{0,1\}^+$ is a factor of $x$ which begins with $b$ and ends with  $a,$ then $u=R_a(v)$ for some factor $v$ of $y.$
Finally, as in the previous case we deduce that $N(y)<N(x)$ and $\delta(y)=\delta(x).$\end{proof}

\begin{rema}\label{ind2} {\em We note that if $x$ is a Sturmian word with $\card(UP(x))<+\infty$ and $N(x)>2,$ then, applying repeatedly ii) and iii) of Lemma~\ref{ind}, we deduce that there exist a Sturmian word $y$ and a morphism $f\in  \{L_0,L_1,R_0,R_1\}^+$ such that $x=f(y),$ $N(y)=2$,  and $\delta(y)=\delta(x).$} 

\end{rema}

\begin{corollary}\label{deltasturm} Let $x\in \{0,1\}^\omega$ be a Sturmian word with $\card(UP(x))<+\infty.$ Then $\delta(x)$ is again Sturmian.
\end{corollary}

\begin{proof}  First suppose $N(x)=2.$ In this case the result follows immediately from  i) of Lemma~\ref{ind}.
Next suppose $N(x)>2. $ Then by Remark~\ref{ind2} there exists a Sturmian word $y$ such that $N(y)=2$ and $\delta(y)=\delta(x).$  Hence $\delta(x)$ is Sturmian.\end{proof} 

\begin{corollary} If $x$ is a standard Sturmian word, then so is $\delta(x)$.
\end{corollary}
\begin{proof} By Lemma \ref{rem:sturm} any standard Sturmian word $x$ has a prefixal factorization. Thus  $\card(UP(x))$ $< \infty$. Moreover, if $x$ is of type $a$, then it begins with the letter $a$.  The proof is then easily obtained by making induction on $N(x)$. If $N(x)=2$ by i) of
Lemma \ref{ind} there exists a Sturmian word $y$ isomorphic to $\delta(x)$ such that $x= L_a(y)$. Since $x$ is a standard Sturmian word, it follows that $y$, as well $\delta(x)$, is a standard Sturmian word. Let us now suppose $N(x)>2$.  By ii) of Lemma \ref{ind} there exists a Sturmian word $y$ such that $N(y)<N(x)$ and $x= L_a(y)$ and $\delta(y)=\delta(x)$. Since $y$ is standard, by induction $\delta(y)=\delta(x)$ is a standard Sturmian word.
\end{proof}

\begin{rema}{\em An infinite word $x$ over the alphabet $\A$ is called {\em episturmian} \cite{DJP} if it is closed under reversal  and $x$ has at most one right special factor of each length.  Corollary \ref{deltasturm}  cannot be extended to episturmian words. Indeed, as we observed in Example \ref{ex:try}, 
in the case of Tribonacci word $x$, which is an episturmian word,  $\delta(x)$ has a unique left special factor of each length and two right special factors of each length, so that $ \delta(x)$ is not episturmian.}    
\end{rema}

\begin{lemma} Let $x\in \{0,1\}^\omega$ be a Sturmian word with $\card(UP(x))<+\infty.$ Let $x=U_0U_1U_2\cdots$ be the unique factorization of $x$ over $UP(x).$ Then there exist distinct unbordered prefixes $U$ and $V$ of $x$ with $|V|<|U|$ such that $\{U_i\,|\,i\geq 0\}=\{U,V\}.$
Moreover $U$ is the longest unbordered prefix of $x$ and $V$ is the longest proper unbordered prefix of $U.$  
\end{lemma}

\begin{proof} Following Corollary~\ref{deltasturm} we have that $\delta(x)$ is a Sturmian word. In particular, writing $x=U_0U_1U_2\cdots $ with $U_i\in UP(x),$ we have $\card(\{U_i\,|\,i\geq 0\}) = 2.$ Thus there exist 
distinct unbordered prefixes $U$ and $V$ of $x$ with $|V|<|U|$ such that $\{U_i\,|\,i\geq 0\}=\{U,V\}.$
Since $U_0$ is the longest unbordered prefix of $x,$ it follows that $U_0=U.$ It remains to show that $V$ is the longest proper unbordered prefix of $U.$ 

Without loss of generality we may assume that $x$ is of type $0.$ We proceed by induction on the length $N(x)$ of the longest unbordered prefix of $x$ to show that $V$ is the longest proper unbordered prefix of $U,$ where $U$ and $V$ are as above.  If $N(x)=2,$ we have that $UP(x)=\{0,01\}$ and the result follows taking $V=0$ and $U=01.$ Next suppose $N(x)>2$ and suppose that the result is true up to $N(x)-1.$  Then by ii) and iii) of Lemma~\ref{ind} it follows that there exists a Sturmian word $y$ such that $x=L_0(y)$ in case $x$ begins with $0,$ and $x=R_0(y)$ in case $x$ begins with $1.$ Moreover, again using ii) and iii) it follows that $N(y)<N(x).$ Hence, if $U'$ denotes the longest unbordered prefix of $y,$ and $V'$ denotes the longest proper unbordered prefix of $U',$ then it follows by induction hypothesis that $y$ factors over $\{U',V'\}.$ If $x$ begins with $0,$ then applying $L_0$ to this factorization we obtain a factorization of $x$ over $\{L_0(U'),L_0(V')\}$ and
 
  hence $U=L_0(U')$ and $V=L_0(V').$  Since $L_0:UP(y)\rightarrow UP(x)$ is a bijection, we deduce that $V$ is the longest proper unbordered prefix of $U.$ 
Similarly, if $x$ begins with $1,$   then applying $R_0$ to this factorization we obtain a factorization of $x$ over $\{R_0(U'),R_0(V')\}$ and hence $U=R_0(U')$ and $V=R_0(V').$  Since $R_0:UP(y)\rightarrow UP(x)\setminus \{1\}$ is a bijection and $U$  begins with $1$ and ends with  $00,$ and the shortest prefix of $U$ beginning with $1$ and ending with $0$ is a proper unbordered prefix of $U,$  we deduce that $V$ is the longest proper unbordered prefix of $U.$ \end{proof}

\begin{lemma}\label{comp}Let $y\in \{0,1\}^\omega$ be a Sturmian word,  $f\in  \{L_0,L_1,R_0,R_1\}^+$ and   set $x=f(y).$  If $x$ is singular, then $y$ is singular. 
Conversely, if  $y$ is singular and of the form $y=u'S'$, where $S'$ is a standard Sturmian word and  $u'\in \{0,1\}^+$ with $|u'|\geq 2$, then $x$ is singular. More precisely, 
there exist a standard Sturmian word $S$ and $u\in \{0,1\}^+$ with $|u'|\leq |u|$ such that $x=uS.$ Moreover, if $f$ admits at least one occurrence of either $L_0$ or $L_1,$ then $|u'|<|u|.$
\end{lemma}
\begin{proof}
 Let us first  suppose that $S'$ is a standard Sturmian word and $y$ is singular and of the form $y=u'S'$ with $u'\in \{0,1\}^+$ and $|u'|\geq 2.$  Since $|u'|\geq 2,$ it follows that either $01S'$ or $10S'$ is a suffix of $y.$ In particular, $u'$ must contain an occurrence of both $0$ and $1.$ Whence for each $a\in \{0,1\}$ we have that $|u'|<|L_a(u')|$ and $|u'|<|R_a(u')|.$

For  $a\in \{0,1\}$ let  $g\in \{L_a,R_a\}$.   Taking $g=L_a$ we have $g(y)=L_a(u'S')=L_a(u')L_a(S')$ and $|u'|<|L_a(u')|.$ Taking $g=R_a$  we have $g(y)=R_a(u'S')=R_a(u')a^{-1}aR_a(S')=R_a(u')a^{-1}L_a(S')$ 
and $|u'|\leq |R_a(u')a^{-1}|.$ Since $L_a(S')$ is a standard Sturmian word, one has that $g(y)$ is singular.  Iterating we deduce that there exist  a standard Sturmian word $S$ and $u\in \{0,1\}^+$ such that $x=f(y)=uS$ and $|u'|\leq |u|$.
Moreover, if $f$ admits at least one occurrence of either $L_0$ or $L_1,$ then $|u'|<|u|.$

Let us now suppose that $x=f(y)$ is singular, i.e., $x=uS$ with $u\in \{0, 1\}^+$ and $S$ a standard Sturmian word. We wish to prove that
$y$ is singular. It suffices to prove the assertion for $f\in \{L_a, R_a\}$, $a\in \{0,1\}$.  Let us first take $f= L_a$. One has that $x$, as well as $S$,  is a Sturmian word of type $a$ beginning with the letter $a$. Hence the word $u$ either ends with the letter $b$ or ends with the letter $a$ immediately followed by the letter $a$. Thus setting $S'= L_a^{-1}(S)$ and $u'=  L_a^{-1}(u)$, one has $y= u'S'$. Since $S'$ is a standard Sturmian word, one has that $y$ is singular.

Let us now take $f= R_a$.  One has that $x$, as well as $S$,  is a Sturmian word of type $a$. We have to consider two cases. Case 1. The word $u$ ends with the letter $a$. We can set $u'= R_a^{-1}(u)$. Since the first letter of $S$ is $a$,  we can write $y= u'aS'$ with
$S= R_a(aS')= aR_a(S')= L_a(S')$. This implies that $S'$ is a standard Sturmian word and that $y$ is singular. Case 2. The word $u$ ends with the letter $b$. Since $S$ begins with the letter $a$, we can write $x= u_1baS''$ with  $S= aS''$ and $u= u_1b$, where the word  $u_1$ if it is different from $\varepsilon$,  ends with the letter $a$.  Setting $u'= R_a^{-1}(u_1)$ one has that  $y= u'bS'$ where $bS= R_a(bS')= baR_a(S')= bL_a(S')$. Thus $S= L_a(S')$ and 
$S'$ is a standard word. From this it follows that $y$ is singular.
\end{proof}

\begin{rema}{\em We note that the assumption in the preceding lemma that $|u'|\geq 2$ is actually necessary. For instance, if $y$ is singular and of the form $y=aS'$ with $a\in \{0,1\}$ and $S'$ standard, then $R_a(y)=aR_a(S')=L_a(S')$ is nonsingular. } 
\end{rema}

\begin{thm}\label{Sturmsing}Let $x$ be a Sturmian word. Then $x\in\P_\infty$ if and only if $x$ is nonsingular.
\end{thm}

\begin{proof} We begin by showing that if $x\notin \P_\infty$ then $x$ is singular. 
For this we prove by induction on $n,$ that if $x\notin \P_n,$ then $x$ is singular. For $n=1,$ we have that if $x\notin \P_1,$ then by Lemma \ref{rem:sturm},  $x=aS$ where $a\in \{0,1\}$ and $S$ is a standard Sturmian word. Thus $x$ is singular.
Next let $n\geq 2,$ and suppose by inductive hypothesis that if $y$ is a Sturmian word and $y\notin \P_{n-1},$ then $y$ is singular. Let $x$ be a Sturmian word with $x\notin \P_n.$ By inductive hypothesis we can suppose $x\in \P_{n-1}.$ In particular, $x\in \P_1.$ If $N(x)=2,$ then by i) of Lemma~\ref{ind} we have $x=L_a(y)$ where $a\in \{0,1\}$ and  $y$ is Sturmian isomorphic to $\delta(x).$ Since $x\notin \P_n$ it follows that $\delta(x)\notin \P_{n-1},$ whence $y\notin \P_{n-1}.$  Hence by induction hypothesis $y$ is singular. Thus we can write $y=uS$ where $u\in \{0,1\}^+$ and $S$ a standard Sturmian word. Thus $x=L_a(y)=L_a(uS)=L_a(u)L_a(S)$ and since $L_a(S)$ is a standard Sturmian word, we deduce that $x$ is singular as required.

If $N(x)>2,$ then, following Remark~\ref{ind2},  there exist a Sturmian word $y$ and a morphism $f\in  \{L_0,L_1,R_0,R_1\}^+$ such that $x=f(y),$ $N(y)=2$, and $\delta(y)=\delta(x).$ 
In particular, $\delta(y)\notin \P_{n-1}$ and hence $y\notin \P_{n}.$ 
Thus applying i) of Lemma~\ref{ind} as above, we deduce that $y$ is singular.  Hence there exist $u\in \{0,1\}^+$ and a standard Sturmian word $S$ such that  $y=uS.$ On the other hand, since $\delta(y)$ is defined (or equivalently $y\in \P_1)$ we must have  $|u|\geq 2.$ 
Thus applying Lemma~\ref{comp} we deduce that $x$ is singular.

Conversely, suppose $x\in \{0,1\}^\omega$ is a Sturmian word of the form $x=uS$ with $u\in \{0,1\}^+$ and $S$ a standard Sturmian word. We will show by induction on $|u|$ that $x\notin \P_\infty.$ If $|u|=1,$ i.e., $x=aS$ for some $a\in \{0,1\},$ then, by Lemma \ref{rem:sturm},  $x\notin \P_1,$
whence $x\notin \P_\infty.$  Next let $n\geq 2$ and assume by induction hypothesis that if $y$ is a Sturmian word of the form $y=u'S'$ where $S'$ is a standard Sturmian word, and $u'\in \{0,1\}^+$ with $|u'|<n,$ then $y\notin \P_\infty.$ Let $x$ be a Sturmian word of the form $x=uS$ with $S$ standard, $u\in \{0,1\}^+$ and $|u|=n.$ Since $n\geq 2$, by Lemma \ref{rem:sturm},  $x$ admits a prefixal factorization so that $\card(UP(x))<\infty$. We consider two cases. 
If $N(x)=2,$ then by i) of Lemma~\ref{ind} we can write $x=L_a(y)$ where $a\in \{0,1\}$ and $y$ is Sturmian and isomorphic to $\delta(x).$ 
Since $x$ is singular, it follows by Lemma~\ref{comp} that $y$ is singular. Thus we can write $y=u'S'$ for some $u'\in \{0,1\}$ and some standard Sturmian word $S'.$ If $|u'|=1,$ then $y\notin \P_1,$ whence $\delta(x)\notin \P_1$ and hence $x\notin \P_\infty.$ If $|u'|\geq 2,$ again by Lemma~\ref{comp} we deduce that $|u'|<|u|$ and hence by induction hypothesis we conclude that $y\notin \P_\infty.$ Hence $\delta(x)\notin \P_\infty$ whence $x\notin \P_\infty.$  

Finally suppose $N(x)>2.$ Then by Remark~\ref{ind2} there exist a Sturmian word $y,$ and a morphism $f\in \{L_0,L_1,R_0,R_1\}^+$ such that $x=f(y)$ and $\delta(x)=\delta(y)$ and $N(y)=2.$ By i) of Lemma~\ref{ind}
there exists a Sturmian word $y'$ isomorphic to $\delta(y)$ such that $y=L_a(y').$  
Thus $x=f\circ L_a(y')$ and $y'$ is isomorphic to $\delta(x).$ Since $x$ is singular, by Lemma~\ref{comp} we deduce that $y'$ is singular. Thus we can write $y'=u'S'$ for some $u'\in \{0,1\}^+$ and some standard Sturmian word $S'.$ If $|u'|=1,$ then $y'\notin \P_1,$ whence $\delta(x)\notin \P_1$ and hence $x\notin \P_\infty.$
If $|u'|\geq 2,$ then by Lemma~\ref{comp} we deduce that $|u'|<n,$ and hence by induction hypothesis we conclude that $y\notin \P_\infty.$ Hence $\delta(x)\notin \P_\infty$ whence $x\notin \P_\infty.$ \end{proof}

\end{document}